\documentclass[11pt, a4paper]{article}

\usepackage{times}
\usepackage{a4wide}
\usepackage{hyperref}
\usepackage[british]{babel}
\usepackage{enumerate, longtable}
\usepackage{amsmath, amscd, amsfonts, amsthm, amssymb, latexsym, comment, stmaryrd, graphicx}
\usepackage{tikz}
\usepackage[T1]{fontenc}
\usepackage[latin1]{inputenc}

\newtheorem{thm}{Theorem}[subsection]
\newtheorem{lem}[thm]{Lemma}
\newtheorem{defi}[thm]{Definition}
\newtheorem{rem}[thm]{Remark}
\newtheorem{ex}[thm]{Example}
\newtheorem{prop}[thm]{Proposition}
\newtheorem{cor}[thm]{Corollary}

\newcommand{\dif}{\mathrm{d}}
\DeclareMathOperator{\Real}{Re}
\newcommand{\CC}{\mathbb{C}}

\newcommand{\NN}{\mathbb{N}}
\newcommand{\PP}{\mathbb{P}}
\newcommand{\QQ}{\mathbb{Q}}
\newcommand{\RR}{\mathbb{R}}
\newcommand{\ZZ}{\mathbb{Z}}

\begin{document}

\title{On conjectures of Sato-Tate and Bruinier-Kohnen}
\author{Sara Arias-de-Reyna\footnote{Universit\'e du Luxembourg,
Facult\'e des Sciences, de la Technologie et de la Communication,
6, rue Richard Coudenhove-Kalergi,
L-1359 Luxembourg, Luxembourg, sara.ariasdereyna@uni.lu}, Ilker Inam\footnote{Uludag University, Deparment of Mathematics, Faculty of Arts and Sciences, 16059 Gorukle, Bursa, Turkey, inam@uludag.edu.tr or ilker.inam@gmail.com}
and Gabor Wiese\footnote{Universit\'e du Luxembourg,
Facult\'e des Sciences, de la Technologie et de la Communication,
6, rue Richard Coudenhove-Kalergi,
L-1359 Luxembourg, Luxembourg, gabor.wiese@uni.lu}}

\maketitle

\begin{abstract}
This article covers three topics.
(1) It establishes links between the density of certain subsets of the set of primes and related subsets of the set of natural numbers.
(2) It extends previous results on a conjecture of Bruinier and Kohnen in three ways: the CM-case is included; under the assumption of the same error term as in previous work one obtains the result in terms of natural density instead of Dedekind-Dirichlet density; the latter type of density can already be achieved by an error term like in the prime number theorem.
(3) It also provides a complete proof of Sato-Tate equidistribution for CM modular forms with
an error term similar to that in the prime number theorem.


\end{abstract}

\section{Introduction}\label{sec:introduction}

A very significant recent result in pure mathematics is the proof of the Sato-Tate
conjecture for non-CM modular eigenforms (even for Hilbert eigenforms)~\cite{BLGHT}.
It asserts that for a normalised ($A(1)=1$) cuspidal eigenform $f = \sum_{n=1}^\infty A(n) q^n$
(with $q = e^{2\pi i z}$) of weight~$k\ge 2$ on $\Gamma_0(N)$ (some $N$)
the normalised coefficients $\frac{A(p)}{2p^{(k-1)/2}} \in [-1,1]$ are equidistributed
with respect to the so-called Sato-Tate measure, when $p$ runs through the set of primes.

The corresponding result for CM forms has been known for a long time and in fact
is quite a simple corollary of the equidistribution of the values of Hecke characters.
In Section~\ref{sec:ST} of this article we include a proof of a form of this result
that additionally provides
an error bound like the one in the prime number theorem (see Theorem~\ref{thm:statements}).
It relies on an error bound for the equidistribution of the values of Hecke characters
given in~\cite{Rajan1998}.

A special case of Sato-Tate equidistribution for non-CM eigenforms shows that the sets of primes
$$\{p \textnormal{ prime } : A(p) > 0\} \textnormal{ and } \{p \textnormal{ prime } : A(p) < 0\}$$
both have natural density equal to~$1/2$.
A conjecture of Bruinier and Kohnen (\cite{BK} and \cite{KLW}) asserts that something
similar should hold for certain half-integral weight modular forms $f = \sum_{n=1}^\infty a(n) q^n$;
namely they conjecture that the sets
$$\{n \in \NN : a(n) > 0\} \textnormal{ and } \{n \in \NN : a(n) < 0\}$$
have the same natural density, namely, half of the natural density of $\{n \in \NN : a(n) \neq 0\}$.
The interest in the distribution of signs is explained by a famous theorem of Waldspurger
relating the squares $(a(t))^2$ for squarefree~$t$
to the critical values of the Hecke L-function of the Shimura lift $F_t$ twisted by an
explicit quadratic character (see~\cite{W1981}); this precisely leaves the sign of $a(t)$
undetermined.

The Bruinier-Kohnen conjecture appears to be quite hard. The main contribution of the previous
work~\cite{IlGa} is the observation that the Shimura lift $F_t$ allows one to utilise
Sato-Tate equidistribution for the coefficients of the integral weight eigenform~$F_t$
in order to compute the densities of the sets of primes
$$\{p \textnormal{ prime } : a(tp^2) > 0\} \textnormal{ and } \{p \textnormal{ prime } : a(tp^2) < 0\}.$$
If the Shimura lift $F_t$ is non-CM, in \cite{IlGa} it is proved that the densities of these
two sets are equal. In this paper we extend this computation to the CM case, see
Theorem~\ref{thm:prime}. It turns out that in the CM case the densities can either
be both $1/2$ or they can be $1/4$ and $3/4$ (see Example~\ref{ex:tunnell}).

In order to study the set of natural numbers
$\{n \in \NN : a(tn^2) > 0\}$ (and similarly for `$<0$')
we set up some general theory, that grew out of analysing the rather ad hoc methods of~\cite{IlGa}.
We now describe this.
Let $\chi: \NN \to \{-1,0,+1\}$ be a multiplicative arithmetic function and define
$$ S_\pm = \{p \textnormal{ prime } : \chi(p) = \pm 1\} \textnormal{ and }
A_\pm = \{n \in \NN : \chi(n) = \pm 1\}.$$
Motivated by the Bruinier-Kohnen conjecture (take $\chi(n)$ to be the sign of $a(tn^2)$
supposing $a(t)>0$), we study the relation between the densities of $S_\pm$ and $A_\pm$. We
were unable to prove any results without the assumption of some error term in the
convergence of the natural density of $S_\pm$. If there is a rather weak error term,
then the sets of primes $S_\pm$ are {\em weakly regular}; if the error term is strong (often implied
by variations of the Riemann Hypothesis), then we obtain {\em regular} sets
(see Definition~\ref{defi:regular}).
Our main results in this abstract context are Propositions \ref{prop:AandS},
\ref{prop:weak-equi}, and~\ref{prop:consequence-Delange}.
In this introduction we do not repeat their precise assertions,
but we explain what they imply for the Bruinier-Kohnen conjecture.

In the case that the Shimura lift $F_t$ has CM, we use the error bound from Theorem~\ref{thm:statements}
in order to obtain the weak regularity of the set $\{p \textnormal{ prime } : a(tp^2) > 0\}$
(and similarly for `$<0$' and `$=0$') and to deduce that
$$\{n \in \NN : a(tn^2) > 0\} \textnormal{ and } \{n \in \NN : a(tn^2) < 0\}$$
have the same {\em Dedekind-Dirichlet density} (see Definition~\ref{defi:Dirichlet}), which
is equal to half the Dedekind-Dirichlet density of $\{n \in \NN : a(tn^2) \neq 0\}$.
Maybe at first sight astonishingly, one obtains this result even in the situation when
the densities of the corresponding sets of primes are not equal.
Under the assumption of a similar error bound in the case that $F_t$ has no CM, one obtains
the same result. This had already been established in~\cite{IlGa} under the
assumption of a stronger error bound. See Remark~\ref{rem:withGRH} for some relation of this
error bound and the Generalised Riemann Hypothesis.
If we assume this stronger error bound (whether $F_t$ is CM or not),
then one can use a result of Delange to derive that the
previous statement even holds in terms of {\em natural density}.
\medskip

The study of the densities of $S_\pm$ and $A_\pm$ is done in Section~\ref{sec:density}.
Our aim there is to give a coherent treatment so that we also recall the relevant definitions.
Section~\ref{sec:ST} is devoted to proving Sato-Tate equidistribution for CM modular
forms (in fact we show slightly more) with an error term as in the prime number theorem.
In the final Section~\ref{sec:BK} the results towards the Bruinier-Kohnen conjecture are derived
from the techniques provided in the other sections.

\subsection*{Acknowledgements}

I.I.\ is supported by The Scientific and Technological Research
Council of Turkey (TUBITAK) and Uludag University Research Project
No: UAP(F) 2012/15. G.W.\ acknowledges partial support by the
priority program 1489 of the Deutsche Forschungsgemeinschaft (DFG).
S.A.\ is partially supported by the project MTM2012-33830 of the
Ministerio de Econom\'ia y Competitividad of Spain.
I.I.\ would like to thank the University of Luxembourg for its hospitality.

The authors would like to thank Juan Arias de Reyna for his
remarks. They also thank Jeremy Rouse for explanations concerning~\cite{RT}.
I.I.\ and G.W.\ are grateful to Winfried Kohnen for interesting discussions.
Thanks are also due to the anonymous referee for helpful suggestions
concerning the presentation of the paper.

\section{Densities and sets of primes}\label{sec:density}

In this section we are concerned with the sets
$$ S_\pm = \{p \textnormal{ prime } : \chi(p) = \pm 1\} \textnormal{ and }
A_\pm = \{n \in \NN : \chi(n) = \pm 1\}$$
for a multiplicative arithmetic function $\chi: \NN \to \{-1,0,+1\}$, as explained
in the introduction.
We found it necessary to assume more than just that $S_\pm$ has a natural density
in order to conclude something about the density of $A_\pm$; namely, we obtain our
results under the assumption that $S_\pm$ is (weakly) regular (see below). We also show
that (weak) regularity is a consequence of a sufficiently good error bound for the convergence
of the natural density of~$S_\pm$.

\subsection{Notions of density}

\begin{defi}
Let $\mathbb{P} \subset \NN$ be the set of all prime numbers.
For a set of primes $S \subseteq \PP$ we make the following definitions:
\begin{itemize}
\itemsep=0cm plus 0pt minus 0pt
\parsep=0.0cm
\item For $x\in \mathbb{R}$, denote $\pi_S(x):=\#\{p\leq x: p\in S\}$. As usual
denote $\pi_\PP$ by~$\pi$.
\item  $P_S(z) := \sum_{p \in S} \frac{1}{p^z}$. This defines a holomorphic function
on $\{\Real(z) > 1\}$.
\item For a multiplicative function $\chi: \NN \to \RR$ we let
$D_\chi(z) := \sum_{n \ge 1} \frac{\chi(n)}{n^z}$ be the corresponding Dirichlet series.
If $|\chi|$ is bounded, it also defines a holomorphic function on $\{\Real(z) > 1\}$.
In particular, $D_1 = \zeta(z)$ is the Riemann-zeta function.
\item A function $\chi: \NN \to \RR$ is said to be
{\em characteristic on~$S$} if $\chi$ is multiplicative
and its restriction to~$\PP$ is the characteristic function of the set~$S$.
\end{itemize}
\end{defi}

The following lemma links the Dirichlet series $D_\chi$ for some $\chi$ that is
characteristic on~$S$ to~$P_S$. This link is the key to relating density statements
on subsets of~$\PP$ to subsets of~$\NN$.

\begin{lem}\label{lem:AD}
Let $\chi: \NN \to \{-1,0,1\}$ be a multiplicative function.
Then on $\{\Real(z) > 1\}$ one has
$$ \log\big(D_\chi(z)\big) = \sum_{p \in \PP} \frac{\chi(p)}{p^z} + g(z),$$
where $g(z)$ is a function that is holomorphic on $\{\Real(z) > 1/2\}$.
In particular, if $\chi$ is characteristic on~$S$, the equality becomes
$\log\big(D_\chi(z)\big) = P_S(z) + g(z)$.
\end{lem}

\begin{proof}
We use the Euler product
$D_\chi(z) = \prod_{p \in \PP}\left(1 + \sum_{n\ge 1} \frac{\chi(p^n)}{p^{nz}}\right)$,
which is absolutely convergent on $\{\Real(z)>1\}$ in the sense that
$\sum_{p \in \PP} \sum_{n\ge 1} \frac{\chi(p^n)}{p^{nz}}$ converges absolutely in this region.

We first treat the following special case. Let $S \subseteq \PP$ and $\chi:\NN \to \{0,1\}$
be multiplicative such that for any prime~$p$
one has $\chi(p^n)=1$ if and only if $p \in S$ and $n=1$. Then the Euler factor
of $D_\chi$ at~$p$ is either $1+\frac{1}{p^z}$ or~$1$, depending on whether $p \in S$ or not.
We take the logarithm of the Euler product
$$ \log D_\chi(z) = \sum_{p \in S} \log(1+\frac{1}{p^z})
= \sum_{p \in S} \frac{1}{p^z} + g(z) \textnormal{ with } g(z) := \sum_{p \in S} \sum_{m \ge 2} \frac{(-1)^{m+1}}{m} \left(\frac{1}{p^z}\right)^m.$$
It is elementary to prove that $g(z)$ defines a holomorphic function on $\{\Real(z)>\frac{1}{2}\}$.

In order to tackle the general case, let $S_\pm :=\{p \in \PP : \chi(p)=\pm 1\}$ and
define the multiplicative functions $\chi_\pm$ on prime powers by $\chi_\pm(p^n) = \begin{cases}
1 & \textnormal{ if } p \in S_\pm \textnormal{ and } n=1,\\ 0 & \textnormal{ otherwise.}\end{cases}$.\\

Define $\Phi(z) := \frac{D_\chi(z) \cdot D_{\chi_-}(z)}{D_{\chi_+}(z)}$.
Then we have on $\{\Real(z)>1\}$
$$ \log(\Phi(z)) = \log(D_\chi(z)) + \log(D_{\chi_-}(z)) - \log(D_{\chi_+}(z))
= \log(D_\chi(z)) + P_{S_-}(z) - P_{S_+}(z) + g(z),$$
where $g(z)$ is holomorphic on $\{\Real(z)>\frac{1}{2}\}$.
On $\{\Real(z)>1\}$ the function~$\Phi$ is described by an absolutely converging product
$\Phi(z) = \prod_{p \in \PP} \Phi_p(z)$,
where $\Phi_p(z)$ satisfies $|1-\Phi_p(z)| \le \frac{20}{p^{2z}}$.
It easily follows that this product converges absolutely on $\{\Real(z)>\frac{1}{2}\}$,
which implies the assertion.
\end{proof}

The density of a set of prime numbers (if it exists) measures its size. There are several notions of density, e.g.\ Dirichlet density and natural density, which in general are distinct.
In a similar way, one can define analogous notions of density for subsets of~$\NN$.
Here we recall the definitions. By the symbol $\lim_{z \to 1^+}$ we denote the limit
defined by letting $z$ tend to~$1$ on the real interval $(1,\infty)$.

\begin{defi}\label{defi:Dirichlet}
Let $S \subseteq \PP$ be a set of primes. The set $S$ is said to have \emph{Dirichlet density}
equal to $\delta(S)$ if the limit
$$ \lim_{z \to 1^+} \frac{\sum_{p\in S}\frac{1}{p^z}}{\sum_{p\in \PP}\frac{1}{p^z}}
= \lim_{z \to 1^+} \frac{\sum_{p\in S}\frac{1}{p^z}}{\log\big( \zeta(z)\big)}
= \lim_{z \to 1^+} \frac{\sum_{p\in S}\frac{1}{p^z}}{\log\big(\frac{1}{z-1}\big)}$$
exists and is equal to~$\delta(S)$.
Moreover, $S$ is said to have \emph{natural density equal to $\mathrm{d}(S)$} if the limit
$$ \lim_{x\rightarrow \infty} \frac{\pi_S(x)}{\pi(x)}$$
exists and is equal to~$\mathrm{d}(S)$.
Let now $A \subseteq \mathbb{N}$ be a subset. It is said to have \emph{Dedekind-Dirichlet density}
$\delta(A)$ if the limit
$$ \lim_{z \to 1^+} \frac{\sum_{n\in A}\frac{1}{n^z}}{\sum_{n\in \NN}\frac{1}{n^z}}
= \lim_{z \to 1^+} \frac{\sum_{n\in A}\frac{1}{n^z}}{\zeta(z)}
= \lim_{z \to 1^+} (z-1)\sum_{n\in A}\frac{1}{n^z}$$
exists and is equal to~$\delta(A)$.
Similarly, $A$ is said to have \emph{natural density} $\mathrm{d}(A)$ if the limit
$$ \lim_{x\rightarrow \infty} \frac{\#\{n\leq x: n\in A\}}{x}$$
exists and is equal to~$\mathrm{d}(A)$.
\end{defi}

The equalities in the statements all follow from Lemma~\ref{lem:AD} and the well-known
fact that the Riemann-zeta function has a simple pole of residue~$1$ at~$1$.
It is well known that if a set of prime numbers $S$ (resp.\ a set of
natural numbers $A$) has a natural density, then it also has a
Dirichlet density (resp.\ a Dedekind-Dirichlet density) and they
coincide.
A function $\chi:\NN \to \{0,1\}$ that is characteristic on $S \subseteq \PP$ links
the set $S$ to the set of natural numbers $A=\{n\in \mathbb{N}: \chi(n)=1\}$.
The following proposition, the proof of which is evident in view of Lemma~\ref{lem:AD},
makes clear the nature of the relation between the Dirichlet density of~$S$
and the Dedekind-Dirichlet density of~$A$.

\begin{prop}\label{prop:compAS}
Let $S$ be a set of primes and $\chi:\mathbb{N}\rightarrow \{0, 1\}$ be a multiplicative function characteristic on~$S$ and let $A=\{n\in \mathbb{N}: \chi(n)=1\}$.
Then the Dirichlet density of $S$, if it exists, equals
\begin{equation*}
\delta(S)=\lim_{z\rightarrow 1^+}\frac{\log D_{\chi}(z)}{\log \zeta(z)}
\end{equation*}
and the Dedekind-Dirichlet density of $A$, if it exists, equals
\begin{equation*}
\delta(A)=\lim_{z\rightarrow 1^+}\frac{D_{\chi}(z)}{\zeta(z)}=\exp\left(\lim_{z\rightarrow 1^+} (\log D_{\chi}(z) - \log \zeta(z))\right).
\end{equation*}
\end{prop}

We now prove a precise relationship between the densities of $A$ and $S$.
This result will be strengthened below
in Proposition~\ref{prop:AandS} under the extra assumption of weak regularity, which
is introduced in the next section.

\begin{prop}\label{prop:deltaAnonzero}
Let $S$ be a set of primes and $\chi:\mathbb{N}\rightarrow \{0, 1\}$ be characteristic on~$S$
and let $A=\{n\in \mathbb{N}: \chi(n)=1\}$.
If $\delta(A) \neq 0$ (in particular, the limit exists), then $\delta(S)=1$.
\end{prop}

\begin{proof}
As $\delta(A) \neq 0$, it follows from Proposition~\ref{prop:compAS}
that
$$\lim_{z\rightarrow 1^+} (\log D_{\chi}(z) - \log \zeta(z))$$
exists. But we have by Lemma~\ref{lem:AD}
\begin{equation}\label{eq:deltaA}
 \log D_{\chi}(z) - \log \zeta(z) = \sum_{p \in S} \frac{1}{p^z} - \sum_{p \in \PP} \frac{1}{p^z} + g(z)
= - \sum_{p \not\in S} \frac{1}{p^z}  + g(z),
\end{equation}
where $g$ is a function that is holomorphic on $\{\Real(z) \ge 1\}$.
This implies the convergence of $\sum_{p \not\in S} \frac{1}{p}$, showing that $\PP \setminus S$
is a set of Dirichlet density~$0$, thus $S$ is of Dirichlet density~$1$.
\end{proof}

\subsection{Regular and weakly regular sets of primes}

\begin{defi}\label{defi:regular}
Let $S \subseteq \PP$ be a set of primes.
We call $S$ \emph{weakly regular} if there is $a \in \RR$ and
a function $g(z)$ which is holomorphic on $\{\Real(z) > 1\}$ and continuous (in particular, finite) on $\{\Real(z) \ge 1\}$
such that
$$ P_S(z) = a \log\big(\frac{1}{z-1}\big) + g(z).$$
As in \cite{Na} (and \cite{IlGa}) we say that $S$ is \emph{regular}\footnote{Added in proof: The notion of a regular set of primes already appeared in \cite{Del56}.}
if the function $g$ is holomorphic on $\{\Real(z) \ge 1\}$.
\end{defi}

Clearly, every regular set~$S$ is weakly regular.
If $S$ is weakly regular, it directly follows that it has a Dirichlet density, namely $\delta(S)=a$.
If $S$ is regular (weakly regular) of density~$0$,
then $P_S$ is holomorphic (continuous) on $\{\Real(z)\ge 1\}$.

\begin{prop}\label{prop:AandS}
Let $S$ be a weakly regular set of primes and
$\chi:\mathbb{N}\rightarrow \{0, 1\}$ be a multiplicative function characteristic
on primes with respect to $S$ and let $A=\{n\in \mathbb{N}: \chi(n)=1\}$.
Then
$$ \delta(A) \neq 0 \;\;\Leftrightarrow\;\; \delta(S) = 1.$$
\end{prop}

\begin{proof}
The direction `$\Rightarrow$' was proved in Proposition~\ref{prop:deltaAnonzero} without the assumption
of weak regularity. Hence, we now assume that $S$ is weakly regular such that $\delta(S)=1$.
It follows that $\PP \setminus S$ is weakly regular of density~$0$,
meaning that $\sum_{p \not\in S} \frac{1}{p^z}$
defines a continuous function on $\{\Real(z) \ge 1\}$.
From Equation~\eqref{eq:deltaA} we get that
$\log D_{\chi}(z) - \log \zeta(z)$ is continuous on $\{\Real(z) \ge 1\}$,
in particular the limit $\lim_{z \rightarrow 1^+}$ exists, whence by
Proposition~\ref{prop:compAS} it follows that $\delta(A)$ exists and is nonzero.
\end{proof}

We next show that sets of primes that have a natural density and additionally satisfy
certain error bounds for the convergence of the limit defining the natural density
are (weakly) regular. In \cite{IlGa}, Proposition~2.2, we proved such a statement.
We will now weaken the assumption on the error term in a way that still allows
to conclude weak regularity instead of regularity.

\begin{prop}\label{prop:weakly-regular}
Let $S$ be a set of primes having natural density $d(S)$.
Let $E(x) := \frac{\pi_S(x)}{\pi(x)} - d(S)$ be the error function.
If the integral $\int_2^\infty \frac{|E(x)|}{x \log(x)}dx$ converges,
then $S$ is a weakly regular set of primes having Dirichlet density $\delta(S) = d(S)$.
\end{prop}

\begin{proof}
The proof follows the proof of \cite{IlGa}, Proposition~2.2, very closely and the
reader is referred there for some of the calculations.
We put $g(x) := E(x)\pi(x)$ and $f(z) = \int_{2}^{\infty} \frac{g(x)}{x^{z+1}} dx$.
Then $P_S(z) = d(S) P(z) + z f(z)$. Hence, it suffices to show that $f$ is
continuous on $\{\Real(z)\ge 1\}$.
We use $\pi(x) < \frac{x}{\log(x)-4}$ for $x > 55$ (by Theorem~29 of~\cite{Rosser})
in order to obtain the estimate
$$ |g(x)|=|E(x)\pi(x)| \le \frac{x |E(x)|}{\log(x)-4}.$$
We now use this to estimate $f(z)$ for $\Real(z)\geq 1$:
$$
|\int_{56}^{\infty}\frac{g(x)dx}{x^{z+1}}|
\le \int_{56}^{\infty}\frac{|g(x)|}{x^{\Real(z)+1}}dx
\le \int_{56}^{\infty}\frac{|E(x)|}{x(\log(x)-4)}dx
\le 2 \int_{56}^{\infty}\frac{|E(x)|}{x\log(x)}dx
$$
The assumption ensures that the last integral is convergent.
Let now $\epsilon > 0$. There is hence some $N$ such that
$|\int_{N}^{\infty}\frac{g(x)dx}{x^{z+1}}| < \epsilon/4$
for any $z$ with $\Real(z)\geq 1$.
Moreover, $f_N(z) := \int_{2}^{N}\frac{g(x)dx}{x^{z+1}}$ is continuous in a neighbourhood of
any such~$z$. In particular, for any $z_1$ with $\Real(z_1)\ge 1$ close enough to $z$
we have $|f_N(z_1)-f_N(z)| < \epsilon/2$. This implies
$|f(z_1)-f(z)| < \epsilon$, as required.
\end{proof}

The following corollary for an explicit error function will be applied in the
situation of CM modular forms in Section~\ref{sec:ST} (see also Proposition \ref{prop:chebotarev} below).

\begin{cor}\label{cor:weakly-regular}
Let $S$ be a set of primes having natural density $d(S)$.
Let $E(x) := \frac{\pi_S(x)}{\pi(x)} - d(S)$ be the error function.
If there are $\alpha > 0$, $C > 0$ and $B>0$ such that
$|E(x)| \le \frac{C}{\log(x)^\alpha}$ for all $x > B$,
then $S$ is a weakly regular set of primes having Dirichlet density $\delta(S) = d(S)$.
\end{cor}

\begin{proof}
Note that the derivative of $- \frac{1}{\alpha \log(x)^{\alpha}}$ is $\frac{1}{x \log(x)^{1+\alpha}}$.
Thus the former is a primitive function for the upper bound of the error term.
As it clearly tends to $0$ for $x \to \infty$, it follows that the assumptions of
Proposition~\ref{prop:weakly-regular} are satisfied.
\end{proof}

The Chebotarev Density Theorem, which plays an essential role in Section~\ref{sec:ST},
provides us with examples of (weakly) regular sets of primes (see Proposition \ref{prop:chebotarev} below), which are used in Section~\ref{sec:BK}.

\begin{defi}
Let $K/\QQ$ be a finite Galois extension with Galois group $G$. We will say that a set $S$ of finite rational primes is a \emph{Chebotarev set for $K/\mathbb{Q}$} if for all $p\in S$, $p$ is unramified in $K/\QQ$ and moreover there exists a subset $C\subseteq G$, invariant under conjugation, such that
$S=\{p\textnormal{ rational prime: } \mathrm{Frob}_{p}\in C\}$,
where $\mathrm{Frob}_p$ denotes a lift to $G$ of the Frobenius element of the residual extension of $K/\QQ$ at a prime~$\mathfrak{p}\vert p$.
\end{defi}

We quote the effective version of the Chebotarev Density Theorem from \cite{Serre1981}.

\begin{thm}[Chebotarev Density Theorem]\label{thm:EffectiveChebotarev}
Let $K/\QQ$ be a finite Galois extension, and let $S$ be a Chebotarev set, which corresponds to $C\subset \mathrm{Gal}(K/\QQ)$. Then the following hold:
\begin{enumerate}[(a)]
\itemsep=0cm plus 0pt minus 0pt
\parsep=0.0cm
\item\label{a-thm:EffectiveChebotarev} For all sufficiently large~$x$,
$\pi_S(x)=\frac{\vert C\vert}{\vert G\vert} \pi(x) + O(x\exp(-c\sqrt{\log(x)}))$
for some constant $c>0$.

\item\label{b-thm:EffectiveChebotarev} If we assume the Riemann Hypothesis for the Dedekind zeta function of $K$, then for all sufficiently large~$x$,
$ \pi_S(x)= \frac{\vert C\vert}{\vert G\vert} \pi(x) + O(x^{1/2}\log(x))$.
\end{enumerate}
\end{thm}

\begin{prop}\label{prop:chebotarev}
Let $K/\QQ$ be a finite Galois extension and $S$ a Chebotarev set.
Then $S$ is weakly regular.
If the Riemann Hypothesis for the Dedekind zeta function of $K$ holds,
then $S$ is regular.\footnote{Added in proof: The assumption of the Riemann Hypothesis is not necessary, see \cite{Ser76}, Proposition~1.5.}
\end{prop}

\begin{proof}
Let $C\subset \mathrm{Gal}(K/\QQ)$ be the set corresponding to $S$. Then by part~(\ref{a-thm:EffectiveChebotarev}) of Theorem~\ref{thm:EffectiveChebotarev}, and taking into account that $\frac{x}{\log(x) + 2}<\pi(x)$ for $x\geq 55$ (see Theorem 29 of \cite{Rosser}), it follows that, for all sufficiently large~$x$,
\begin{equation*}
\left| \frac{\pi_S(x)}{\pi(x)}- \frac{\vert C\vert}{\vert G\vert}\right|\leq c_1 \frac{x\exp(-c_2\sqrt{\log(x)})}{\pi(x)}\leq c_1(\log x + 2)\exp(-c_2\sqrt{\log(x)})
\end{equation*}
for some positive constants $c_1$ and $c_2$. It is clear that this quantity is less than or equal to $\frac{c_3}{\log(x)^{\alpha}}$ for sufficiently large $x$, where $\alpha$ and $c_3$ are any positive constants. Thus by Corollary~\ref{cor:weakly-regular} we can conclude that $S$ is weakly regular.

If we assume that the Dedekind zeta function of $K$ satisfies the Riemann Hypothesis,
then part~(\ref{b-thm:EffectiveChebotarev}) of Theorem~\ref{thm:EffectiveChebotarev} yields
\begin{equation*}
\left| \frac{\pi_S(x)}{\pi(x)}- \frac{\vert C\vert}{\vert G\vert}\right|\leq c_1 \frac{x^{1/2}\log(x)}{\pi(x)}\leq c_1(\log x + 2)\log(x)x^{-1/2}
\end{equation*} for all big enough values of~$x$, where $c_1$ is some positive constant. Proposition~2.2 of~\cite{IlGa} implies that $S$ is regular.
\end{proof}

\subsection{An application: weak regularity yields Dedekind-Dirichlet density}

In this section we derive an equidistribution result,
which will allow us to establish our results towards
the Bruinier-Kohnen conjecture in Section~\ref{sec:BK}.

\begin{prop}\label{prop:weak-equi}
Let $\mathbb{P}=\mathbb{P}_{=0}\cup\mathbb{P}_{>0}\cup\mathbb{P}_{<0}$ be a partition of the set of all primes into three weakly regular sets such that  $\mathbb{P}_{=0}$ is of Dirichlet density~$0$ and the Dirichlet density of $\mathbb{P}_{<0}$ is not zero.
Let $\psi:\mathbb{N}\rightarrow \{0, 1, -1\}$ be a multiplicative arithmetic function such that, for every prime $p$, $\psi(p)=0$ (resp. $\psi(p)=1$, $\psi(p)=-1$) if and only if $p\in\mathbb{P}_{=0}$ (resp.\ $p\in\mathbb{P}_{>0}$, $p\in\mathbb{P}_{<0}$).

Then  $\{n: \psi(n)> 0\}$ and $\{n: \psi(n)< 0\}$ have a Dedekind-Dirichlet density, which for
both is $1/2$ of the Dedekind-Dirichlet density of $\{n:\psi(n)\not=0\}$.
\end{prop}

\begin{proof}
Let us record first that the set $\{n:\psi(n)\not=0\}$ indeed has a positive Dedekind-Dirichlet
density by Proposition~\ref{prop:AandS}, that is, the limit
\begin{equation}\label{eq:Dpsi1}
\lim_{z \to 1^+} (z-1)\sum_{n \in \NN}\frac{|\psi(n)|}{n^z} = d
\end{equation}
exists with $0<d\le 1$.
Lemma~\ref{lem:AD} yields
$$\log(D_\psi(z)) = \sum_{p \in \PP} \frac{\psi(p)}{p^z} + g(z)
= \sum_{p \in \PP_{>0}}\frac{1}{p^z} - \sum_{p \in \PP_{<0}}\frac{1}{p^z} + g(z),$$
where $g(z)$ is a function that is holomorphic on $\{\Real(z) > 1/2\}$.
Using the definition of weak regularity for the sets $\PP_{>0}$ and $\PP_{<0}$, we obtain
$$\log(D_\psi(z)) = a \log(\frac{1}{z-1}) + h(z),$$
or, equivalently,
\begin{equation*}
D_\psi(z) = \frac{1}{(z-1)^a} \exp(h(z)),
\end{equation*}
where $a$ is $\delta(\PP_{>0}) - \delta(\PP_{<0})$, which is strictly less than~$1$
by assumption, and $h(z)$ is a function that is continuous on $\{\Real(z) \ge 1\}$.
Taking the exponential yields
\begin{equation}\label{eq:Dpsi2}
D_\psi(z) = \sum_{n \in \NN, \psi(n)=1} \frac{1}{n^z} - \sum_{n \in \NN, \psi(n)=-1} \frac{1}{n^z}
= \frac{1}{(z-1)^a} \phi(z),
\end{equation}
where $\phi(z)=\exp(h(z))$ is also continuous on $\{\Real(z) \ge 1\}$.
Adding Equations \eqref{eq:Dpsi1} and~\eqref{eq:Dpsi2} yields
$$ \lim_{z \to 1^+} (z-1)\big( 2\sum_{n \in \NN, \psi(n)=1} \frac{1}{n^z} \big) = d,$$
which is the claimed formula.
\end{proof}

\subsection{Towards natural density}

In this section we show that regularity of density~$1$ for a set $S \subseteq \PP$ suffices
to conclude that the set of natural numbers corresponding to a function that is
characteristic on~$S$ has a positive {\em natural} density, and not only a Dedekind-Dirichlet
density, whose existence was shown in Proposition~\ref{prop:AandS}.
In fact, one sees that a slightly weaker assumption than regularity works, however,
we are unable to prove that weak regularity is enough.

\begin{prop}\label{prop:middle-regular}
Let $S \subseteq \PP$ be a set of primes of density~$1$ and let $\chi:\NN \to \{0,1\}$
be a multiplicative function characteristic on~$S$.
We assume that $S$ satisfies the following condition (which is implied by regularity
but not weak regularity):
\begin{quote}
The function
$$g(z) := \sum_{p \in S} \frac{1}{p^z} - \log(\frac{1}{z-1}), $$
which is holomorphic on $\{\Real(z)>1\}$, is once differentiable
at $z= 1$ in the sense that $\varphi(z) := \frac{g(z)-g(1)}{z-1}$
can be continued to a continuous function on $\{\Real(z)\ge 1\}$.
\end{quote}
Then there are $0 < a \in \RR$ and a continuous function $h$ on $\{\Real(z) \ge 1\}$
such that
$$ D_\chi(z) = \frac{a}{z-1} + h(z).$$
\end{prop}

\begin{proof}
Lemma~\ref{lem:AD} yields
$$ \log D_\chi(z) = P_S + g_1(z),$$
where $g_1(z)$ is holomorphic on $\{\Real(z)\ge 1\}$.
Combining this with the assumption yields
\begin{equation}\label{eq:middle-regular}
\log D_\chi(z) = \log \frac{1}{z-1} + k(z),
\end{equation}
where $k(z)$ is continuous on $\{\Real(z)\ge 1\}$ and satisfies that the
difference quotient
$\psi(z) := \frac{k(z)-k(1)}{z-1}$
also defines a continuous function on $\{\Real(z)\ge 1\}$.
An elementary calculation yields
$$ \exp(k(z)) = \exp(k(1)) + (z-1) \exp(k(1)) \left(\sum_{n=1}^\infty \frac{(z-1)^{n-1}\psi(z)^n}{n!} \right).$$
Note that the series on the right hand side defines a continuous function on $\{\Real(z)\ge 1\}$.
Putting $a=\exp(k(1))$ and combining the previous calculation
with Equation~\eqref{eq:middle-regular} finishes the proof.
\end{proof}

We now use the following version of the famous Wiener-Ikehara theorem taken
from~\cite{Ko} in order to conclude the existence of natural density instead of
`only' Dedekind-Dirichlet density in some cases.

\begin{thm}[Wiener-Ikehara]\label{thm:wiener-ikehara}
Let $(a_n)_n$ be a sequence of real numbers satisfying:
\begin{enumerate}
\itemsep=0cm plus 0pt minus 0pt
\parsep=0.0cm
\item $a_n\geq 0$ for all $n\in \mathbb{N}$.

\item $\sum_{n\geq 1} \frac{a_n}{n^z}$ converges for $\Real z>1$.

\item There exists $a\in \mathbb{C}$, $g(z)$ continuous on $\{\Real z\geq 1\}$ such that
\begin{equation*}
 \sum_{n\geq 1}\frac{a_n}{n^z}=\frac{a}{z-1} + g(z) \text{ for all }z\in \{\Real z >1\}.
\end{equation*}

\item There exists $C>0$ such that, for all $n\in \mathbb{N}$, $\sum_{k=1}^n a_k\leq Cn$.
\end{enumerate}
Then
\begin{equation*}
 \lim_{n\rightarrow \infty}\frac{\sum_{k=1}^n a_k}{n}=a.
\end{equation*}
\end{thm}

The hard assumption in our case is 3; it is a strong form of Dedekind-Dirichlet density.
The conclusion of Proposition~\ref{prop:middle-regular} is that this strong form
holds under the assumptions of that proposition.
Thus we obtain from the Wiener-Ikehara Theorem~\ref{thm:wiener-ikehara}:

\begin{cor}\label{cor:natural-regular}
Assume the set-up of Proposition~\ref{prop:middle-regular}.
Let $A := \{n \in \NN : \chi(n) \neq 0\}$.
Then $A$ has a natural density, which is equal to~$a > 0$.
\end{cor}

\subsection{An application: regularity yields natural density}

In this section, we utilise the following theorem of Delange in order to strengthen
Proposition~\ref{prop:weak-equi} to natural density under the assumption of regularity.

\begin{thm}[\cite{Delange}]\label{thm:delange}
Let $f:\mathbb{N}\rightarrow \mathbb{C}$ be a multiplicative arithmetic function, satisfying:
\begin{enumerate}
\itemsep=0cm plus 0pt minus 0pt
\parsep=0.0cm
\item $\vert f(n)\vert \leq 1$  for all $n\in \mathbb{N}$.
\item There exists $a\in \mathbb{C}$, $a\not=1$ such that
$\lim_{x\rightarrow \infty}\frac{\sum_{p\leq x, p\text{ prime}}f(p)}{\pi(x)} = a$.
\end{enumerate}
Then
\begin{equation*}
 \lim_{x\rightarrow \infty}\frac{\sum_{n\leq x}f(n)}{x}=0.
\end{equation*}
\end{thm}

\begin{prop}\label{prop:consequence-Delange} Let $\mathbb{P}=\mathbb{P}_{=0}\cup\mathbb{P}_{>0}\cup\mathbb{P}_{<0}$ be a partition of the set of all primes into three sets with natural density, such that  $\mathbb{P}_{=0}$ is regular  of density~$0$ and the natural density of $\mathbb{P}_{<0}$ is not zero.
Let $\psi:\mathbb{N}\rightarrow \{0, 1, -1\}$ be a multiplicative arithmetic function such that, for every prime $p$, $\psi(p)=0$ (resp. $\psi(p)=1$, $\psi(p)=-1$) if and only if $p\in\mathbb{P}_{=0}$ (resp.\ $p\in\mathbb{P}_{>0}$, $p\in\mathbb{P}_{<0}$).

Then  $\{n: \psi(n)> 0\}$ and $\{n: \psi(n)< 0\}$ have a natural density, which for both
is $1/2$ of the natural density of $\{n:\psi(n)\not=0\}$.
\end{prop}

\begin{proof} We want to apply Delange's Theorem~\ref{thm:delange}
with $f= \psi$. The first condition is trivially satisfied. Concerning the second condition, note that
\begin{equation*}
 \sum_{p\leq x, p\text{ prime}}f(p)=\#\{p\leq x: p\in \mathbb{P}_{>0}\} - \#\{p\leq x: p\in \mathbb{P}_{<0}\},
\end{equation*}
thus
\begin{equation*}\lim_{x\rightarrow \infty}\frac{\sum_{p\leq x, p\text{ prime}}f(p)}{\pi(x)}=
\lim_{x\rightarrow \infty}\left(\frac{\#\{p\leq x: p\in \mathbb{P}_{>0}\}}{\pi(x)} -\frac{\#\{p\leq x: p\in \mathbb{P}_{<0}\}}{\pi(x)}\right) \end{equation*} exists because both $\mathbb{P}_{>0}$ and $\mathbb{P}_{<0}$ have natural density by hypothesis, and since the natural density of $\mathbb{P}_{<0}$ is not zero, the limit does not equal $1$. Therefore the second condition is also satisfied.
As a conclusion, we obtain that
\begin{equation*}
 \lim_{x\rightarrow \infty}\frac{\sum_{n\leq x}\psi(n)}{x}=0.
\end{equation*}
In other words,
\begin{equation}\label{eq:parasumar1}
 \lim_{x\rightarrow \infty}\frac{\#\{n\leq x: \psi(n)>0\} - \#\{n\leq x: \psi(n)<0\}}{x}=0.
\end{equation}
Note that $\vert \psi\vert$ is characteristic on $\PP\setminus\PP_{=0}$,
thus by Corollary~\ref{cor:natural-regular} the set
$\{n:\psi(n)\neq 0\}$ has a natural density, call it~$a$.
Therefore
\begin{equation}\label{eq:parasumar2}
 \lim_{x\rightarrow \infty}\frac{\#\{n\leq x: \psi(n)>0\} + \#\{n\leq x: \psi(n)<0\}}{x} =a.
\end{equation}
Adding and substracting \eqref{eq:parasumar1} and \eqref{eq:parasumar2} we obtain that both limits
\begin{equation*}
 \lim_{x\rightarrow \infty}\frac{\#\{n\leq x: \psi(n)> 0\}}{x}\text{ and }\lim_{x\rightarrow \infty}  \frac{\#\{n\leq x: \psi(n)<0\}}{x}
\end{equation*}
exist, and by \eqref{eq:parasumar1} they coincide.
\end{proof}

\section{Sato-Tate Conjecture with error terms}\label{sec:ST}

In this section we collect some known results about the distribution of Fourier coefficients
of modular eigenforms and for the CM case we provide a proof of an error bound similar
to the one in the prime number theorem.

\subsection{Statements}

A sequence $(x_n)_{n \in \NN} \subseteq [-1,1]$ is said to be
{\em $\mu$-equidistributed} (see \cite{KN}, Chapter~3, Def.~1.1)
for a nonnegative regular normed Borel measure~$\mu$ on $[-1,1]$
if for all continuous functions $f: [-1,1] \to \RR$
$$ \lim_{N \to \infty} \frac{1}{N} \sum_{n=1}^N f(x_n) = \int_{-1}^1 f d\mu.$$
Let $k, N\in \mathbb{N}$, and let $f\in S_k(\Gamma_0(N))$ be a normalised cuspidal modular eigenform. Let $f(z)=\sum_{n=1}^{\infty}a_n q^n$ be the Fourier expansion of $f$ at infinity. Since $f$ has trivial character, $a_n\in \RR$ for all $n\in \NN$ and by the Ramanujan-Petersson bounds
$ \vert a_p\vert \leq 2p^{(k-1)/2}$.
It is then natural to study the distribution of $\frac{a_p}{2p^{(k-1)/2}}$ in the interval $[-1, 1]$ as $p$ runs through the prime numbers. It turns out that these values are equidistributed, but the distribution is quite different according to whether the modular eigenform has complex multiplication or not.
The \emph{Sato-Tate measure}, denoted $\mu_{\mathrm{ST}}$, and the
\emph{Sato-Tate measure in the CM case}, denoted $\mu_{\mathrm{CM}}$,
are the probability measures defined on $[-1, 1]$ by the following expressions:
for every Borel-measurable set $A$:
\begin{equation*}
\mu_{\mathrm{ST}}(A) :=\frac{2}{\pi} \int_A \sqrt{1 - t^2} \dif t \textnormal{ and }
\mu_{\mathrm{CM}}(A) :=\frac{1}{2} \delta_0(A) +  \frac{1}{2\pi}\int_A \frac{1}{\sqrt{1 - t^2}} \mathrm{d}t,
\end{equation*}
where $\delta_0$ denotes the Dirac measure at zero.

The Sato-Tate conjecture, now a theorem (cf.~\cite{BLGHT}), asserts that if $f$ has no CM, the real numbers $\frac{a_p}{2p^{(k-1)/2}}$ are equidistributed in $[-1, 1]$ with respect to the measure $\mu_{\mathrm{ST}}$ as $p$ runs through the primes.
Instead of equidistribution in the sense of its definition, we are rather interested
in the set of primes defined by the condition
\begin{equation*}
 S_I:=\{p\in \PP: \frac{a_p}{2p^{(k-1)/2}}\in I\},
\end{equation*}
where $I\subseteq [-1, 1]$ is a subinterval (open, closed or half-open) of $[-1,1]$.
The Sato-Tate conjecture implies that $S_I$ has a natural density equal to $\mu_{\mathrm{ST}}(I)$.
If $f$ has CM it follows from the equidistribution of the values of Hecke characters that
$S_I$ has a natural density equal to $\mu_{\mathrm{CM}}(I)$.
Theorem~1.2 in Chapter~3 of~\cite{KN} can be used to show that also in this case the
values $\frac{a_p}{2p^{(k-1)/2}}$ are $\mu_{\mathrm{CM}}$-equidistributed in the sense
of the definition; but note that from $\mu_{\mathrm{CM}}$-equidistribution alone one may not
conclude anything on the natural density of $S_I$ if the boundary of $I$ has positive mass.

In Section~\ref{sec:BK} we need some knowledge of the speed of the convergence of the quotient
\begin{equation}\label{eq:quotient}
 \frac{\#\{p\in \PP: p\leq x \textnormal{ and }p\in S_I\}}{\pi(x)}
\end{equation} to its limit.
In the CM case the following theorem provides such an error term, which
follows from the work of Hecke on equidistribution of the values of Hecke characters.
Since we did not find a reference with the precise statement as above
(the result of equidistribution of the values of Hecke characters with
an error term only seems to have been published in 1998 in~\cite{Rajan1998}),
we include a proof in this section with and without assuming the Generalised
Riemann Hypothesis.

\begin{thm}\label{thm:statements}
Let $k, N\in \mathbb{N}$, and let $f\in S_k(\Gamma_0(N))$ be a normalised cuspidal modular eigenform with Fourier expansion $f(z)=\sum_{n=0}^{\infty}a_n q^n$.
Assume that $f$ has CM.
\begin{enumerate}[(a)]
\item\label{a-thm:statements} Then there exists a constant $c_1>0$ (depending only on $f$) such that, for all subintervals (open, closed, or half-open) $I\subseteq [-1, 1]$,
\begin{equation*}
\# \{p \textnormal{ prime}: p\leq x,  \frac{a_p}{2p^{(k-1)/2}} \in I\}
=\mu_{\mathrm{CM}}(I)\pi(x) + O(x\exp(-c_1 \sqrt{\log x})),
\end{equation*}
where the implied constant depends only on~$f$.
\item\label{b-thm:statements} Assume the Generalised Riemann Hypothesis for all powers of the Hecke character underlying~$f$
(see Section~\ref{sec:HeckeCharacters}).
Then for all subintervals (open, closed, or half-open) $I\subseteq [-1, 1]$
and all $\epsilon>0$
\begin{equation*}
\# \{p \textnormal{ prime}: p\leq x,  \frac{a_p}{2p^{(k-1)/2}} \in I\}
=\mu_{\mathrm{CM}}(I)\pi(x) + O(x^{1/2+\epsilon}).
\end{equation*}
\end{enumerate}
\end{thm}

Very recently, the following theorem covering the case of non-CM modular forms of squarefree
level was proved.

\begin{thm}[Rouse, Thorner]\label{thm:RT}
Let $k, N\in \mathbb{N}$ with squarefree~$N$, and let $f\in S_k(\Gamma_0(N))$ be a normalised cuspidal modular eigenform with Fourier expansion $f(z)=\sum_{n=0}^{\infty}a_n q^n$.
Assume that $f$ does not have CM.
Assume that all the symmetric power L-functions of $f$ are automorphic and satisfy
the Generalised Riemann Hypothesis.
Then for all subintervals (open, closed, or half-open) $I\subseteq [-1, 1]$,
\begin{equation*}
\# \{p \textnormal{ prime}: p\leq x,  \frac{a_p}{2p^{(k-1)/2}} \in I\}
=\mu_{\mathrm{ST}}(I)\pi(x) + O(x^{3/4}).
\end{equation*}
\end{thm}

\begin{proof}
This is an easy consequence of Theorem~1.3 of~\cite{RT}.
\end{proof}

\begin{rem}\label{rem:withGRH} For non-CM modular forms~$f$ we have not found in the literature any unconditional result for the error term in the convergence of the quotient \eqref{eq:quotient} to the natural density of $S_I$.

When $f$ is attached to an elliptic curve $E/\QQ$, if we assume analytic continuation, functional equation, and the Generalised Riemann Hypothesis for the $L$-function attached to the $m$-th symmetric power of $E$ for every $m\in \NN$, then V.\ Kumar Murty (cf.~\cite{M1985}) states the error bound
\begin{equation}\label{eq:error}
\# \{p \textnormal{ prime}: p\leq x,  \frac{a_p}{2p^{(k-1)/2}} \in I\}
=\mu_{\mathrm{ST}}(I)\pi(x) + O(x^{\frac{1}{2}+ \varepsilon}).
\end{equation}

Akiyama and Tanigawa proved a converse of this statement. Namely, they prove that, if formula \eqref{eq:error} holds for an elliptic curve $E/\QQ$ without CM, then the Generalised Riemann Hypothesis holds for the $L$-function $L(s, E)$ (cf.\ Theorem 2 of~\cite{AT}).

Jeremy Rouse informed us that he expects that a statement similar to Theorem~\ref{thm:RT} should hold in non-squarefree level.
\end{rem}

\begin{cor}\label{cor:statements}
\begin{enumerate}[(a)]
\item In the set-up of Theorem~\ref{thm:statements} part~(\ref{a-thm:statements}) the set
$\{p \textnormal{ prime}:  \frac{a_p}{2p^{(k-1)/2}} \in I\}$ is weakly regular.
\item In the set-up of Theorem~\ref{thm:statements} part~(\ref{b-thm:statements}) or
of Theorem~\ref{thm:RT} the set
$\{p \textnormal{ prime}:  \frac{a_p}{2p^{(k-1)/2}} \in I\}$ is regular.
\end{enumerate}
\end{cor}

\begin{proof}
This follows respectively from Corollary~\ref{cor:weakly-regular} and \cite{IlGa}, Proposition~2.2.
\end{proof}

\begin{rem}
If $f$ is a Hecke eigenform with real Fourier coefficients $a_n$,
a natural question to study is the distribution of the signs of the $a_n$
as $n$ runs through the set of natural numbers. For $f$ of half-integral weight,
this study is the content of the Bruinier-Kohnen conjecture (see Section~\ref{sec:BK}).
Here we include the easier case of $f\in S_k(\Gamma_0(N))$ of integral weight.
We can combine the results of the previous sections with those in this section in order
to address this question.
Define the sets $\PP_{>0}$ (resp. $\PP_{<0}$, $\PP_{=0}$,  $\PP_{\not=0}$) as the set of primes such that $a_p>0$ (resp.\ $a_p<0$, $a_p=0$, $a_p\neq0$).
\begin{enumerate}[(a)]
\item Assume that $f$ has CM. By Corollary~\ref{cor:statements}, the set $\PP_{=0}$ is weakly regular of natural density equal to~$1/2$, and the sets $\PP_{>0}$ and $\PP_{<0}$ are both weakly regular of density~$1/4$.
Consider the character $\chi:\NN\rightarrow \{0, 1\}$ defined as $\chi(n)=1$ if and only if $a_n\neq 0$. 
We can apply Proposition~\ref{prop:AandS} and conclude that $\{n\in \NN: a_n\neq 0\}$ cannot have a positive Dedekind-Dirichlet density.

\item Assume now that $f$ satisfies the assumptions of Theorem~\ref{thm:RT}.
Then by Corollary~\ref{cor:statements} the sets $\PP_{=0}$, $\PP_{>0}$, and $\PP_{<0}$
are regular of natural density equal to $0$, $1/2$, $1/2$, respectively.
Thus by Proposition~\ref{prop:consequence-Delange}, $\{n \in \NN : a_n>0\}$ and
$\{n \in \NN : a_n<0\}$ have the same natural density, which equals $1/2$ of the natural density of $\{n\in \NN: a_n\neq0\}$.
\end{enumerate}
\end{rem}

We devote the rest of this section to explaining in detail how the equidistribution
of the values of the Hecke characters implies Theorem~\ref{thm:statements}.

\subsection{Hecke characters}\label{sec:HeckeCharacters}

We first set up some general notation that will below be specialised to imaginary quadratic fields.
Let $K$ be a number field of degree $g$, and let $\mathcal{O}_K$ the ring of integers of $K$. As usual, we denote $g=r_1 + 2r_2$, where $r_1$ is the number of real embeddings of $K$, and $2r_2$ is the number of complex embeddings. We will write the embeddings as $\tau_1, \dots, \tau_{g}:K\rightarrow \mathbb{C}$, where the first $r_1$ are the real embeddings, and $\tau_{\nu}$ is the complex conjugate of $\tau_{\nu + r_2}$ for all $\nu\in \{r_1, \dots, r_1 + r_2\}$.
For any fractional ideal $\mathfrak{a}$ of $K$, we denote by $v_{\mathfrak{p}}(\mathfrak{a})$ the exponent of $\mathfrak{p}$ in the factorisation of $\mathfrak{a}$ into prime ideals.
Let $I$ be the group of fractional ideals of $K$, and let us fix an integral ideal $\mathfrak{m}$ (not necessarily a maximal ideal) of the ring of integers of $K$.

\begin{defi}
Let $a, b\in K^{\times}$.
We say that $a\equiv b\,\mathrm{mod}^{\times} \mathfrak{m}$ if, for all $\mathfrak{p}\vert \mathfrak{m}$,
$v_{\mathfrak{p}}(a-b)\geq v_{\mathfrak{p}}(\mathfrak{m})$.
\end{defi}

\begin{defi}
Let $I(\mathfrak{m}):=\{\mathfrak{a}\in I: (\mathfrak{a}, \mathfrak{m})=1\}$. A character
$\xi: I(\mathfrak{m})\rightarrow \{z\in \mathbb{C}: \vert z\vert =1\}$
is called a \emph{Hecke character} mod $\mathfrak{m}$ if there exists a set of pairs of real numbers $\{(u_{\nu}, v_{\nu})$, $\nu=1, \dots, r_1 + r_2\}$, satisfying:
\begin{itemize}
\itemsep=0cm plus 0pt minus 0pt
\parsep=0cm
\item $u_{\nu}\in \mathbb{Z}$; moreover $u_\nu\in \{0, 1\}$ if $\nu\leq r_1$.
\item $\sum_{\nu=1}^{r_1 + r_2} v_{\nu}=0$.
\item  For all $a\in K^{\times}$ such that  $a\equiv 1 \,\mathrm{mod}^\times \mathfrak{m}$, 
$\xi((a))=\prod_{\nu=1}^{r_1 + r_2}\left(\frac{\tau_{\nu}(a)}{\vert \tau_{\nu}(a)\vert}\right)^{u_{\nu}} \vert \tau_{\nu}(a) \vert^{iv_\nu}$.
\end{itemize}
\end{defi}

The values of the Hecke characters are equidistributed on the unit circle: the probability that they lie on an arc is proportional to the length of the arc. This fact was already known to Hecke (cf.\ \cite{Hecke1920}). The explicit version we state below are Theorem~1 and Proposition~4 of~\cite{Rajan1998}.
We use the standard notation $\pi_K(x)=\#\{\mathfrak{p}\textnormal{ prime ideal of }K: \mathrm{Norm}_K(\mathfrak{p})\leq x\}$.

\begin{thm}\label{thm:equidistr}
Let $K$ be a number field, $\mathfrak{m}$ an integral ideal of $K$ and $\xi:I(\mathfrak{m})\rightarrow \{z\in \CC: \vert z\vert=1\}$ a Hecke character of infinite order.
\begin{enumerate}[(a)]
\item\label{a-thm:equidistr}
There exists a constant $c_1>0$ (depending only on $K$) such that, for all $\alpha,\beta \in [-\pi, \pi]$ with $\beta \le \alpha$
\begin{multline*}
\#\{\mathfrak{p}\textnormal{ prime ideal of }\mathcal{O}_K : (\mathfrak{p}, \mathfrak{m})=1, N_K(\mathfrak{p})\leq x, \mathrm{arg}(\xi(\mathfrak{p}))\in [\beta,\alpha)\}
\\ =\frac{1}{2\pi}(\alpha-\beta)\pi_K(x) + O(x\exp(-c_1\sqrt{\log x})),
\end{multline*} where the implicit constant depends only on $K$.

\item\label{b-thm:equidistr} Assume in addition that the L-functions of all powers of $\xi$
satisfy the Generalised Riemann Hypothesis. Then for all $\epsilon>0$ and all
$\alpha,\beta \in [-\pi, \pi]$ with $\beta \le \alpha$,
\begin{multline*}
\#\{\mathfrak{p}\textnormal{ prime ideal of }\mathcal{O}_K : (\mathfrak{p}, \mathfrak{m})=1, N_K(\mathfrak{p})\leq x, \mathrm{arg}(\xi(\mathfrak{p}))\in [\beta,\alpha)\}
\\ =\frac{1}{2\pi}(\alpha-\beta)\pi_K(x) + O(x^{1/2+\epsilon}),
\end{multline*}
\end{enumerate}
\end{thm}

We may replace the interval $[\beta, \alpha)$ by $[\beta, \alpha]$, $(\beta, \alpha]$ or $(\beta, \alpha)$ in the statement of Theorem~\ref{thm:equidistr}.

\begin{rem}\label{rem:equidistr}
It is straightforward to translate Theorem~\ref{thm:equidistr} into the following statement on the distribution of the projections of $\xi(\mathfrak{p})$ to the real axis:
for all subintervals $I \subseteq [-1, 1]$ (open, closed, or half-open) one has
\begin{multline*}
\#\{\mathfrak{p}\textnormal{ prime ideal of }\mathcal{O}_K : (\mathfrak{p}, \mathfrak{m})=1, N_K(\mathfrak{p})\leq x, \mathrm{Re}(\xi(\mathfrak{p}))\in I\}
\\ =\left(\frac{1}{\pi}\int_I\frac{1}{\sqrt{1-t^2}}\dif t\right)\pi_K(x) + O(x\exp(-c_1\sqrt{\log x})).
\end{multline*}
Under the assumption of part~(\ref{b-thm:equidistr}) the error term is $O(x^{1/2+\epsilon})$.
\end{rem}

Assume now that $K=\QQ(\sqrt{d})$ is an imaginary quadratic field. In this case $g=2$, $r_1=0$ and $r_2=1$. Thus in this particular case, given an integral ideal $\mathfrak{m}$ of $K$ as above, a Hecke character is a character $\xi: I(\mathfrak{m})\rightarrow \{z\in \mathbb{C}: \vert z\vert =1\}$
such that, for all $a\in K^{\times}$ such that $a\equiv 1\,\mathrm{mod}^\times \mathfrak{m}$,
it holds that
$\xi((a))=\left(\frac{\tau(a)}{\vert \tau(a)\vert}\right)^{u}$
for some $u\in \ZZ$, which we may assume positive by changing the choice of the embedding $\tau$ by its conjugate, if necessary.
The next result (cf.\ Theorem 4.8.2 of~\cite{Miyake}) attaches CM modular forms to such characters:

\begin{thm}\label{Hecke}
Let $K$, $\mathfrak{m}$, $\xi$, $u$ as above. Assume $u\not=0$. Then the expression
 \begin{equation}\label{eq:Hecke}
f(z):=  \sum_{\mathfrak{a}} \xi(\mathfrak{a}) N_{K/\mathbb{Q}}(\mathfrak{a})^{u/2} q^{N_{K/\mathbb{Q}}(\mathfrak{a})}
 \end{equation}
defines a modular form $f\in S_{u+1}(N, \chi)$, where $\mathfrak{a}$ runs through all integral ideals of $K$ with $(\mathfrak{a}, \mathfrak{m})=1$, $N=\vert d \vert \mathrm{Norm}_{K}(\mathfrak{m})$ and where $\chi$ is the Dirichlet character defined as
\begin{equation}\label{eq:nebentypus}
 \chi(m)=\left(\frac{d}{m}\right)\xi((m))\mathrm{sgn}(m)^u \text{ for all }m\in \mathbb{Z}.
\end{equation}
\end{thm}

Conversely, any modular form with CM arises in this way from some Hecke character of an imaginary quadratic field (cf.\ \cite{Ribet1976}, Thm.~4.5).

\subsection{Equidistribution of Fourier coefficients of CM modular forms}

Assume now that we have a normalised eigenform $f\in S_k(\Gamma_0(N))$ such that $f$ has CM by the imaginary quadratic field $K$. Let $\xi$ be the Hecke character that gives rise to $f$ as in
Theorem~\ref{Hecke}. Then the Fourier expansion of $f$ looks like Equation~\eqref{eq:Hecke}. In particular, for all primes $p\nmid N$, we have
\begin{equation*}
a_p=\begin{cases} \xi(\mathfrak{p}_1) N_K(\mathfrak{p}_1)^{\frac{k-1}{2}} +  \xi(\mathfrak{p}_2) N_K(\mathfrak{p}_2)^{\frac{k-1}{2}} &\text{ if }(p)=\mathfrak{p}_1\mathfrak{p}_2\text{ with }\mathfrak{p}_1\not=\mathfrak{p}_2;\\
      0 &\text{ if } (p) \text{ is inert in }K.
     \end{cases}
\end{equation*}
Since $f$ has trivial nebentypus,  Equation~\eqref{eq:nebentypus} implies that $\xi((p))=1$ whenever $p$ splits in~$K$. Thus if $(p)=\mathfrak{p}_1 \mathfrak{p}_2$, then $\xi(\mathfrak{p}_1)$ and $\xi(\mathfrak{p}_2)$ are complex conjugates. Therefore
\begin{equation}\label{eq:split}
 \frac{a_p}{2p^{(k-1)/2}}=\mathrm{Re}(\xi(\mathfrak{p}_1)).
\end{equation}
We introduce the notation
$$ \pi_{K/\QQ,\mathrm{split}}(x):=\#\{p \textnormal{ rational prime}: p\leq x\text{ and }(p) \textnormal{ splits in }K/\mathbb{Q}\}$$
and similarly $\pi_{K/\QQ,\mathrm{inert}}(x)$ and $\pi_{K/\QQ,\mathrm{ram}}(x)$.

\begin{lem}\label{lem:cuentaK}
We have that
\begin{equation*}
\#\{\mathfrak{p}\textnormal{ prime ideal of }\mathcal{O}_K : \textnormal{Norm}_{K/\QQ} \le x
\textnormal{ and } \mathfrak{p}/(\mathfrak{p} \cap \ZZ)
\textnormal{ is not split }\} = O(\sqrt{x})
\end{equation*}
and $\pi_K(x)=2\pi_{K/\mathbb{Q},\mathrm{split}}(x) + O(\sqrt{x})$.
\end{lem}

\begin{proof}
The number of elements in the set of the first claim is clearly at most
$\#\{p\text{ prime }:  p\leq \sqrt{x} \} = O(\sqrt{x})$.
The second claim follows from the equality
\begin{equation*}
\pi_K(x)=2\pi_{K/\mathbb{Q}, \mathrm{split}}(x) + \pi_{K/\mathbb{Q}, \mathrm{inert}}(\sqrt{x}) + \pi_{K/\mathbb{Q}, \mathrm{ram}}(x)
\end{equation*}
and the fact that only finitely many primes ramify in $K/\QQ$.
\end{proof}

\begin{proof}[Proof of Theorem \ref{thm:statements}]
We only prove part~(\ref{a-thm:statements}), since the arguments in part~(\ref{b-thm:statements})
are entirely analogous. Let $I\subseteq[-1, 1]$ be a subinterval.
We want to count how many primes $p$ satisfy that $\frac{a_p}{2p^{(k-1)/2}}\in I$.
We count the split and the inert primes separately and start with the inert ones:
\begin{multline*}
\# \{p \textnormal{ prime inert in }K: p\leq x, p \nmid  N, \frac{a_p}{2p^{(k-1)/2}} \in I\}\\ =\begin{cases}
\# \{p \textnormal{ prime inert in }K: p\leq x, p \nmid  N\} &\text{ if }0\in I;\\
0 &\text{ if }0\not\in I.\end{cases}
\end{multline*}
This implies
\begin{equation}\label{eq:pinert}
\# \{p \textnormal{ prime inert in }K: p\leq x, p \nmid  N, \frac{a_p}{2p^{(k-1)/2}} \in I\}
=\frac{1}{2}\delta_0(I)\pi(x) + O(x\exp(-c \sqrt{\log x})),
\end{equation}
where we have used that
$\# \{p \text{ prime inert in }K: p\leq x, p \nmid  N\}=\frac{1}{2}\pi(x) + O(x\exp(-c \sqrt{\log x}))$
for some constant $c>0$, which follows from part~(\ref{a-thm:EffectiveChebotarev}) of Theorem~\ref{thm:EffectiveChebotarev}.
The split primes are counted using Remark~\ref{rem:equidistr} and Lemma~\ref{lem:cuentaK}
as follows:
\begin{equation}\label{eq:psplit}\begin{aligned}
&\# \{p \textnormal{ prime split in }K: p\leq x, p \nmid  N, \frac{a_p}{2p^{(k-1)/2}} \in I\}\\
=&\frac{1}{2} \# \{\mathfrak{p} \textnormal{ prime of } \mathcal{O}_K :
\textnormal{Norm}_{K/\QQ}(\mathfrak{p}) \le x, \mathfrak{p}/(\mathfrak{p} \cap \ZZ)
\textnormal{ is split }, \Real(\xi(\mathfrak{p})) \in I \}\\
=&\frac{1}{2} \# \{\mathfrak{p} \textnormal{ prime of } \mathcal{O}_K :
\textnormal{Norm}_{K/\QQ}(\mathfrak{p}) \le x, \Real(\xi(\mathfrak{p})) \in I \}+ O(\sqrt{x})\\
=&\frac{1}{2}\left(\frac{1}{\pi} \int_I \frac{1}{\sqrt{1-t^2}} \dif t\right)\pi_K(x) + O(x\exp(-c \sqrt{\log x}))\\
=&\left(\frac{1}{\pi} \int_I \frac{1}{\sqrt{1-t^2}} \dif t\right)
\pi_{K/\mathbb{Q}, \mathrm{split}}(x) + O(x\exp(-c \sqrt{\log x}))\\
=&\frac{1}{2}\left(\frac{1}{\pi} \int_I \frac{1}{\sqrt{1-t^2}} \dif t\right)
\pi(x) + O(x\exp(-c \sqrt{\log x}))\\
\end{aligned}
\end{equation}
for some constant $c>0$. The theorem follows by adding Equations
\eqref{eq:pinert} and~\eqref{eq:psplit}.
\end{proof}

\section{Application to the Bruinier-Kohnen Conjecture}\label{sec:BK}

\subsection{Equidistribution of signs of half-integral weight modular forms - the prime case}\label{sec:pc}

In this section, we state an analog of the Bruinier-Kohnen sign equidistribution conjecture for the family $\{a(tp^2)\}$ where $t$ is a squarefree number such that $a(t)\neq 0$ and $p$ runs through the primes for a half-integral weight modular form whose Shimura lift is without CM or with CM. The proof will be carried out in Section \ref{sec:perturbedseq}. Furthermore we will give some properties of these coefficient sets. Note that the following theorem is an improvement of Theorems 4.1 and~4.2 of~\cite{IlGa}.

We start by summarising some known facts about half-integral weight modular forms and the Shimura lift. Let $k \geq 2$. According to Shimura~\cite{Shi} and Niwa~\cite{Niwa}, if $f$ is a Hecke eigenform of weight $k+1/2$ with Fourier expansion $f=\sum_{n=1}^{\infty}a(n)q^n \in S_{k+1/2}(N, \chi)$ then there is a corresponding modular form $F_t \in S_{2k}(N/2, \chi^2)$ for fixed $t \geq 1$ squarefree such that $a(t) \neq 0$, named the {\em Shimura lift of $f$ with respect to~$t$}, such that the $T_{n^2}$-Hecke eigenvalue on~$f$ agrees with the $T_n$-Hecke eigenvalue on~$F_t$. For $k=1$ suppose that $f$ is contained in the orthogonal complement with respect to the Petersson scalar product of the subspace
$S_{k+1/2}(N,\chi)$ generated by unary theta functions as in~\cite{BK}.
The Fourier expansion of $F_t$ is given by $F_t(z)=\sum_{n \geq 1}^{}A_t(n)q^n$ where
\begin{equation}\label{eq:1}
A_t(n):=\sum_{d|n}^{} \chi_{t,N}(d)d^{k-1}a(\frac{tn^2}{d^2}),
\end{equation}
where $\chi_{t,N}$ denotes the character $\chi_{t,N}(d):=\chi(d) \left(\frac{(-1)^{k}N^2t}{d} \right)$.
Moreover, the Fourier coefficients are multiplicative in the sense
\begin{equation}\label{eq:mult}
a(tm^2)a(tn^2)=a(t)a(tm^2n^2)
\end{equation}
for $(n,m)=1$.
If $F_t$ has CM, then let $\mu$ denote $\mu_{\mathrm{CM}}$, otherwise put $\mu=\mu_{\mathrm{ST}}$.
We assume throughout that $\chi$ is trivial or quadratic and that $f$ has real coefficients.
This implies that $F_t$ also has real coefficients.

The following is our main theorem about the distribution of signs of the coefficients
$a(tp^2)$, when $p$ runs through the primes.
In the statement we understand by an equality of two Dirichlet characters the equality
of the underlying primitive characters (i.e.\ we allow them to differ at finitely many primes).

\begin{thm}\label{thm:prime}
Assume the set-up above and define the set of primes
$$\PP_{> 0}:=\{p \in \PP : a(tp^2)  > 0 \}$$
and similarly $\PP_{< 0}$ and $\PP_{= 0}$ (depending on $f$ and~$t$).

\begin{enumerate}[(a)]

\item\label{a-thm:prime} If $F_t$ has no complex multiplication then the sets $\PP_{> 0}$ and $\PP_{< 0}$ have natural density $1/2$ and the set $\PP_{= 0}$ has natural density $0$.

\item\label{b-thm:prime}

\begin{enumerate}[(i)]

\item\label{bi-thm:prime} If $F_t$ has complex multiplication and $\chi_{t,N} = 1$ then the set $\PP_{=0}$ has natural density equal to zero, and the sets $\PP_{>0}$ and $\PP_{<0}$ have natural densities, respectively $1/4$ and $3/4$ if $a(t)>0$ and, respectively $3/4$ and $1/4$ if $a(t)<0$.

\item\label{bii-thm:prime} If $F_t$ has complex multiplication and $\chi_{t,N} = \delta$, where $\delta$ is the quadratic Dirichlet
character corresponding to the imaginary quadratic field by which $f$ has CM, then the set $\PP_{=0}$ has natural density equal to zero, and the sets $\PP_{>0}$ and $\PP_{<0}$ have natural densities, respectively $3/4$ and $1/4$ if $a(t)>0$ and, respectively $1/4$ and $3/4$ if $a(t)<0$.

\item\label{biii-thm:prime} If $F_t$ has complex multiplication and $\chi_{t,N} \not\in \{1, \delta\}$ then the set $\PP_{=0}$ has natural density equal to zero, and the sets $\PP_{>0}$ and $\PP_{<0}$ have the same natural density which is equal to $1/2$.
\end{enumerate}

\item\label{c-thm:prime} If $F_t$ has no complex multiplication then we additionally assume that there are $C>0$ and $\alpha >0$ such that for all subintervals $[a,b] \subseteq [-1,1]$ one has
$$ \left| \frac{ \# \{p\le x \textnormal{ prime }| \; \frac{A_t(p)}{a(t)2p^{k-1/2}} \in [a,b]\}}{\pi (x)}
 - \mu([a,b]) \right| \le \frac{C}{\log(x)^\alpha}.$$

Then the sets $\PP_{>0}$, $\PP_{<0}$, and $\PP_{=0}$ are weakly regular sets of primes.

\item\label{d-thm:prime} Assume here that there are $C>0$ and $\alpha >0$ such that for all subintervals $[a,b] \subseteq [-1,1]$ one has
$$ \left| \frac{ \# \{p\le x \textnormal{ prime }| \; \frac{A_t(p)}{a(t)2p^{k-1/2}} \in [a,b]\}}{\pi (x)}
 - \mu([a,b]) \right| \le \frac{C}{x^\alpha}$$
(note that this condition is satisfied if $F_t/a(t)$ fulfills the assumptions of Theorem~\ref{thm:RT}, see also Remark~\ref{rem:withGRH}).
Then the sets $\PP_{>0}$, $\PP_{<0}$, and $\PP_{=0}$ are regular sets of primes.

\end{enumerate}
\end{thm}

\begin{ex}\label{ex:tunnell}
Consider the elliptic curve defined by the equation
\begin{equation*}
 y^2=x^3-x.
\end{equation*}
This elliptic curve has conductor $32$ and has CM by $\ZZ[i]$.
Let $F = \sum_{n=1}^\infty A(n) q^n \in S_2(32)$ be the associated cuspidal eigenform. We have that, for all $p\equiv -1 \pmod 4$, $A(p)=0$, that is, $F$ has CM by~$\QQ(i)$. In \cite{Tunnell}, Tunnell has shown that there exist modular forms $f_1\in S_{3/2}(128)$ (trivial character)
and $f_2\in S_{3/2}(128, \chi_2)$, where $\chi_2=\left(\frac{2}{\cdot}\right)$, such that their Shimura lifts with $t=1$ coincide with $F$.

\begin{itemize}

\item For $f_1$, we have $\chi_{1,128}(p)=\left(\frac{-1\cdot 4}{p}\right)$, which coincides
with the character by which $F$ has CM.
Thus, $\PP_{>0}$ has natural density $3/4$ and $\PP_{<0}$ has natural density~$1/4$.

\item For $f_2$, we have $\chi_{1, 128}(p)=\left(\frac{-2}{p}\right)$, which is different
from the trivial character and the character by which $F$ has CM.
In this case the densities of $\PP_{>0}$ and  $\PP_{<0}$ coincide and they are equal to $1/2$.
\end{itemize}
\end{ex}

\begin{rem}\label{rem:perturbed}
\begin{enumerate}[(a)]
\item\label{rem:perturbed:a}
For fixed squarefree~$t$ such that $a(t)\neq 0$ we use the notation:
$$A(p):=\frac{a(tp^2)}{a(t)2p^{k-1/2}} \textnormal{ and }
  B(p):=\frac{A_t(p)}{a(t)2p^{k-1/2}}.$$
Note that Equation~\eqref{eq:1} implies
\begin{equation}\label{eq:BpAp}
A(p) = B(p) - \frac{\chi_{t,N}(p)}{2\sqrt{p}}.
\end{equation}

The main point in our approach is that we view the sequence $A(p)$ as a `perturbed' version
of the sequence~$B(p)$.

\item\label{rem:perturbed:b}
We remark that `small' perturbations preserve the property of a sequence to be equidistributed.
More precisely, let $\mu$ be a nonnegative regular normed Borel measure on $[-1,1]$
and $(x_n)_{n \in \NN} \subseteq [-1,1]$ be a $\mu$-equidistributed
sequence. Let $(y_n)_{n \in \NN} \subseteq [-1,1]$ be a sequence
such that
$$\lim_{n \to \infty} |x_n - y_n| = 0.$$
Then also $(y_n)_{n \in \NN}$ is $\mu$-equidistributed.

This follows from a straight forward calculation using the definition of
$\mu$-equidistribution and the compactness of~$[-1,1]$.

\item\label{rem:perturbed:c}
Returning to our set-up of modular forms, we first remark
that the set~$S$ of primes~$p$ such that
$\frac{a(tp^2)}{2a(t)p^{k-1/2}} \not\in [-1,1]$ has natural density~$0$
(this is an easy consequence of Theorem~\ref{thm:abstract} below).

Part~(\ref{rem:perturbed:b}) above together with Equation~\eqref{eq:BpAp} thus implies that the elements
$\left(\frac{a(tp^2)}{2a(t)p^{k-1/2}}\right)_{p \in \PP\setminus S}$
are $\mu$-equidistributed.

We stress that equidistribution of $\frac{a(tp^2)}{2a(t)p^{k-1/2}}$ is not
enough to imply equidistribution of signs if the measure has points of positive
mass (like $\mu_{\mathrm{CM}}$). See for instance Example~\ref{ex:tunnell}.
This is the reason why we are not only interested in equidistribution in the sense
of the definition, but, are studying the limits
$\lim_{N \to \infty} \frac{\#\{n \le N : x_n \in I\}}{N}$
for all intervals~$I$, even those having a boundary of positive measure.
\end{enumerate}
\end{rem}

\subsection{Densities of perturbed sequences}\label{sec:perturbedseq}

In this section we provide a treatment of an abstract setting modeled on the relation
between coefficients of half-integral and integral weight modular forms under the
Shimura lift (see, in particular, Remark~\ref{rem:perturbed}),
and we will use it to prove Theorem~\ref{thm:prime}.

\begin{thm}\label{thm:abstract}
Let $\chi$ be a Dirichlet character of order dividing~$2$.
Let $B: \PP \to \RR$ be a map and define $A: \PP \to \RR$ by the formula
$A(p) := B(p) - \frac{\chi(p)}{y\sqrt{p}}$ for some $0 \neq y \in \RR$.
Let $D = \{x_1,\dots,x_n\} \subset [-1,1]$.
For any $I \subseteq [-1,1]$ define
$$S_I := \{p \in \PP : B(p) \in I\} \textnormal{ and }
T_I' := \{p \in \PP : A(p) \in I, B(p) \not\in D\}.$$
Let $f:(-1,1) \to \RR_{\ge 0}$ be an integrable function and $w_1,\dots,w_n \ge 0$.
Define a measure on $[-1,1]$ by
$\mu(I) = \int_I f(t) dt + \sum_{i=1}^n w_i \delta_{x_i}(I)$, where $\delta_{x_i}$ is the
Dirac measure at the point~$x_i$, for any Borel measurable subset $I \subseteq [-1,1]$.
Assume that $\mu([-1,1])=1$ and that for all intervals $I \subseteq [-1,1]$ (open, closed or half-open)
the set $S_I$ has natural density~$\mu(I)$.

\begin{enumerate}[(a)]
\item\label{a-thm:abstract} Then for any interval $I \subseteq [-1,1]$ (open, closed or half-open),
the set $T_I'$ has natural density $\int_I f(t) dt$.

\item\label{b-thm:abstract}  Assume that there are $m \in \NN_{\ge1}$ and $M >0$
such that for all $\epsilon>0$ small enough the integrals $|\int_{1-\epsilon}^1 f(t) dt|$ and
 $|\int_{-1}^{-1+\epsilon} f(t) dt|$ are bounded above by $M \epsilon^{1/m}$.
Assume moreover that there is a function $E(x)$ tending to~$0$ as $x \to \infty$ such that
for all intervals $I \subseteq [-1,1]$
$$ \left| \frac{\pi_{S_I}(x)}{\pi(x)} - \mu(I) \right| \le E(x).$$
Then for any interval $I \subseteq [-1,1]$ (open, closed or half-open), there is $C > 0$ such that
for all big enough~$x$
$$ \left|\frac{\pi_{T_I'}(x)}{\pi(x)} - \int_I f(t) dt\right| \le C \cdot \left(E(x) + \frac{1}{x^{1/(2m+2)}}\right).$$
\end{enumerate}
\end{thm}

\begin{proof}
For any interval $I \subseteq [-1,1]$ define
$$S_I'(x) := \{p \in \PP : B(p) \in I, B(p) \not\in D\} = S_{I \setminus D}.$$
By assumption the set $S_I'$ has natural density $\int_I f(t) dt$. Let $a$ be the start point
and $b$ the end point of~$I$.
Let $\epsilon > 0$ be small enough. For all $p > \frac{1}{y^2\epsilon^2}$ one has
$\left|\frac{\chi(p)}{y\sqrt{p}}\right| < \epsilon$. One observes the inequalities
\begin{equation}\label{eq:ineq-dens}
    \frac{\pi_{S'_{[a+\epsilon,b-\epsilon]}}(x) - \pi(1/(y\epsilon)^2)}{\pi(x)}
\le \frac{\pi_{T_{[a,b]}'}(x)}{\pi(x)}
\le \frac{\pi_{S'_{[\max\{-1,a-\epsilon\},\min\{1,b+\epsilon\}]}}(x) + \pi(1/(y\epsilon)^2)}{\pi(x)}.
\end{equation}

(\ref{a-thm:abstract}) From Equation~\eqref{eq:ineq-dens} we obtain the inequalities
$$\int_{a+\epsilon}^{b-\epsilon} f(t) dt \le \liminf_{x \to \infty} \frac{\pi_{T_{[a,b]}'}(x)}{\pi(x)}
\textnormal{ and }
\limsup_{x \to \infty} \frac{\pi_{T_{[a,b]}'}(x)}{\pi(x)} \le \int_{\max\{-1,a-\epsilon\}}^{\min\{1,b+\epsilon\}} f(t) dt.$$
Letting $\epsilon$ tend to~$0$ we obtain
$$ \limsup_{x \to \infty} \frac{\pi_{T_{[a,b]}'}(x)}{\pi(x)} \le \int_{a}^{b} f(t) dt
\le \liminf_{x \to \infty} \frac{\pi_{T_{[a,b]}'}(x)}{\pi(x)},$$
implying the result.

(\ref{b-thm:abstract}) Equation~\eqref{eq:ineq-dens} yields
\begin{multline*}
- \int_{a}^{a+\epsilon} f(t) dt - \int_{b-\epsilon}^b f(t) dt - (n+1)E(x) - \frac{\pi(1/(y\epsilon)^2)}{\pi(x)}
\le \frac{\pi_{T_{[a,b]}'}(x)}{\pi(x)} - \int_a^b f(t) dt\\
\le \int_{\max\{-1,a-\epsilon\}}^a f(t) dt + \int_b^{\min\{1,b+\epsilon\}} f(t) dt + (n+1)E(x) + \frac{\pi(1/(y\epsilon)^2)}{\pi(x)},
\end{multline*}
which is valid for all (small enough) $\epsilon>0$ and all (big enough)~$x$.
Using the assumptions we obtain
$$ \left| \frac{\pi_{T_{[a,b]}'}(x)}{\pi(x)} - \int_a^b f(t) dt\right|
\le 2M\epsilon^{1/m} + \frac{\pi(1/(y\epsilon)^2)}{\pi(x)} + (n+1)E(x).$$
We may (and do) assume that $E(x) \ge \frac{1}{x^{1/(2m+2)}}$ for large
enough~$x$.
Let $\epsilon := E(x)^m$. One finds
$$\frac{\pi(1/(y\epsilon)^2)}{\pi(x)}
= \frac{\pi(1/(y^2E(x)^{2m}))}{\pi(x)}
\sim \frac{\log(x)}{y^2 \cdot E(x)^{2m} \cdot \log(1/(y^2 E(x)^{2m})) \cdot x}
\le C \cdot E(x)$$
for $x$ big enough and suitable $C > 0$. Thus we obtain the claimed inequality.
\end{proof}

\begin{rem}\label{rem:appl}
For the applications below we remark that for $I \subseteq [-1,1]$ we have
\begin{align*}
\{p \in \PP : A(p) \in I\}
& = T_I' \sqcup \bigsqcup_{i=1}^n \{ p \in \PP : B(p)=x_i, A(p) \in I\}\\
& = T_I' \sqcup \bigsqcup_{i=1}^n \big(\{ p \in \PP : B(p)=x_i\} \cap
                                   \{p \in \PP : x_i - \frac{\chi(p)}{y\sqrt{p}} \in I\}\big).\\
\end{align*}
Note that we have
$$ \{p \in \PP : x_i - \frac{\chi(p)}{y\sqrt{p}} \in I\} =
\begin{cases}
\textnormal{finite set}   & \textnormal{ if } x_i \not\in \overline{I},\\
\PP \setminus\textnormal{finite set} & \textnormal{ if } x_i \in \overset{\circ}{I},\\
\end{cases}$$
where $\overline{I}$ denotes the closure and $\overset{\circ}{I}$ the interior of~$I$.
If $I=[x_i,b]$ with $b > x_i$, then moreover
$$ \{p \in \PP : x_i - \frac{\chi(p)}{y\sqrt{p}} \in I\} =
   \{p \in \PP : \chi(p) = - \frac{y}{|y|}\} \setminus\textnormal{finite set},$$
and analogously for $I=[a,x_i]$ with $a < x_i$,
$$ \{p \in \PP : x_i - \frac{\chi(p)}{y\sqrt{p}} \in I\} =
   \{p \in \PP : \chi(p) = \frac{y}{|y|}\}\setminus\textnormal{finite set}.$$
The same formulas hold if the intervals are open or half-open.
In particular, for any interval $I$, the set
$\{p \in \PP : x_i - \frac{\chi(p)}{y\sqrt{p}} \in I\}$ has a density, which is one of $0,\frac{1}{2},1$.
\end{rem}

\begin{proof}[Proof of Theorem \ref{thm:prime}]
We use the notation introduced in Remark~\ref{rem:perturbed}~(\ref{rem:perturbed:a}).
\begin{itemize}
\item[(\ref{a-thm:prime})] See \cite{IlGa}, Theorem~4.1.

\item[(\ref{b-thm:prime})]
Assume that $F_t$ has complex multiplication.
Put $D=\{0\}$, $f=\frac{1}{2\pi}\frac{1}{\sqrt{1-t^2}}$, $I=(0,1]$ and $J=[-1, 0)$.
Take
$$S_I(x):=\{p \in \PP : B(p) \in I\}, \;\;\; T_I':=\{p \in \PP : A(p) \in I, B(p) \neq 0\}$$
and similarly
$$S_J(x):=\{p \in \PP : B(p) \in J\}, \;\;\; T_J':=\{p \in \PP : A(p) \in J, B(p) \neq 0\}.$$
The sets $S_I$ and $S_J$ have natural densities, respectively $\mu_{CM}(I)$ and $\mu_{CM}(J)$ by Theorem~\ref{thm:statements}, so that we can apply Theorem~\ref{thm:abstract}.
For simplicity we assume $a(t)>0$. The arguments in the other case $a(t)<0$ are exactly the same. We have $\{p \in \PP : A(p)>0 \}=\PP_{>0}$. By Remark~\ref{rem:appl}, we conclude that
\begin{align}\label{eq:3}
\PP_{>0}
&=T_I' \sqcup \left( \{p \in \PP : B(p)=0 \} \cap \{p \in \PP : \frac{-\chi_{t,N}(p)}{2\sqrt{p}} \in I \}   \right)\\
\PP_{<0}
&=T_J' \sqcup \left( \{p \in \PP : B(p)=0 \} \cap \{p \in \PP : \frac{-\chi_{t,N}(p)}{2\sqrt{p}} \in J \}   \right)\label{eq:3a}
\end{align}

In order to compute $d(\PP_{>0})$, we compute the sum of $d(T_I')$ and the density of the intersection, and similarly for $d(\PP_{<0})$. We have $d(T_I')=\mu(I)=\frac{1}{4}$ and $d(T_J')=\mu(J)=\frac{1}{4}$ by Theorem~\ref{thm:abstract}.

\begin{itemize}
\item[(\ref{bi-thm:prime})]
Assume that $\chi_{t,N}=1$ (recall that by an equality of Dirichlet characters we understand that the underlying primitive characters agree). In this case, since the intersection in Equation \eqref{eq:3} is finite and therefore has density $0$, we conclude that the set $\PP_{>0}$ has density~$1/4$. Similarly, the intersection in Equation \eqref{eq:3a} has density~$1/2$, therefore $\PP_{<0}$ has density~$3/4$. It is clear that the set $\PP_{=0}$ has natural density equal to zero.

\item[(\ref{bii-thm:prime})] We will do the same computation as above. Note that in this case we have
$$\{p \in \PP : B(p)=0 \}=\{p \in \PP : \delta(p)=-1 \}$$
up to finitely many primes. These sets have natural density $1/2$. Suppose that $\chi_{t,N}=\delta$. Then the density of the intersection in Equation \eqref{eq:3} is $1/2$ by Remark \ref{rem:appl}. So we conclude that $\PP_{>0}$ has natural density~$3/4$. Similarly, from Equation \eqref{eq:3a} we obtain that $\PP_{<0}$ has natural density $1/4$.

\item[(\ref{biii-thm:prime})] Suppose that $\chi_{t,N} \neq 1,\delta$. By Chebotarev's theorem, the intersections in Equations \eqref{eq:3} and \eqref{eq:3a} have natural density $1/4$. So we conclude that the sets $\PP_{>0}$ and $\PP_{<0}$ have the same natural density, which is equal to $1/2$.
\end{itemize}

\item[(\ref{c-thm:prime})]
By assumption in the non-CM case and by Theorem~\ref{thm:statements} in the CM case, we have for all intervals $I \subseteq [-1,1]$
$$ \left| \frac{\pi_{S_I}(x)}{\pi(x)} - \mu (I) \right| \leq \frac{C}{\log(x)^{\alpha}}.    $$
For the CM-case we need
$\int_{1-\epsilon}^1 f(t) dt = \int_{1-\epsilon}^1 \frac{1}{2\pi\sqrt{1-t^2}} dt\le \sqrt{\epsilon}$,
as a simple calculation shows.
The corresponding check in the non-CM case is trivial since the density function of the measure
is continuous on~$[-1,1]$.
Thus, in both cases Theorem~\ref{thm:abstract}~(\ref{b-thm:abstract}) yields
$$ \left| \frac{\pi_{T'_I}(x)}{\pi(x)} - \int_{I}^{} f(t)dt  \right| \leq  \frac{\tilde{C}}{\log(x)^{\alpha}}  $$
for some $\tilde{C}>0$, where $f$ is the density function in the CM or non-CM case.
Corollary~\ref{cor:weakly-regular} shows that $T'_I$ is weakly regular.

Since $\PP_{=0}=T'_I$ for $I=[0,0]=\{0\}$, $\PP_{\geq 0}=T'_{[0,1]}$ and $\PP_{>0}=T'_{(0,1]}$ in the non-CM case, it follows that the sets $\PP_{=0}$, $\PP_{\geq 0}$ and $\PP_{>0}$ are weakly regular set of primes. By a  similar argument, it is easily seen that the the sets $\PP_{\leq 0}$ and $\PP_{<0}$ are weakly regular sets of primes.

Let us consider the CM case. Then $\PP_{=0}$ is a weakly regular set of primes, since $T'_{[0,0]}$ is. We have to show that the intersections in Equations \eqref{eq:3} and \eqref{eq:3a} are weakly regular sets, since finite disjoint unions of weakly regular sets are weakly regular.

So, assume that $\chi_{t,N}=1$. In this case the intersection in Equation \eqref{eq:3} is finite and therefore weakly regular of density~$0$.
Since the set $\{p \in \PP : B(p)=0 \}$ is weakly regular of density $1/2$ by Proposition~\ref{prop:chebotarev} and $\{ p \in \PP : \frac{-\chi_{t,N}(p)}{2\sqrt{p}} \in [-1,0) \}$ is $\PP$ (except for a finite set), we conclude that the intersection in Equation~\eqref{eq:3a} is weakly regular of density~$1/2$.

For the case $\chi_{t,N}=\delta$, the intersection in Equation \eqref{eq:3} is $\{ p \in \PP : B(p)=0 \}$ up to finitely many primes, hence weakly regular of density~$1/2$ by Proposition \ref{prop:chebotarev}. The intersection in Equation \eqref{eq:3a} is finite and hence also weakly regular.

In the last case $\chi_{t,N} \neq 1, \delta$, the intersections in Equations \eqref{eq:3} and \eqref{eq:3a} are weakly regular of density $1/4$ by Proposition \ref{prop:chebotarev}.

\item[(\ref{d-thm:prime})] Similar arguments as in part~(\ref{c-thm:prime}) prove the assertions, using Proposition~2.2 of~\cite{IlGa}
instead of Corollary~\ref{cor:weakly-regular} and replacing weak regularity by regularity
throughout.

\end{itemize}
\end{proof}

\subsection{Equidistribution of signs of half-integral weight modular forms - the general case}

We now apply the results from Section~\ref{sec:density} and Theorem~\ref{thm:prime} to
obtain an equidistribution statement for the signs of $a(tn^2)$ for $n \in \NN$,
as well as many subsets of~$\NN$.

In order to give a uniform description of the results, let $\chi: \NN \to \{0,1\}$ be
a multiplicative arithmetic function such that $\chi(p)=1$ for all primes~$p \in \PP$.
Then define $\NN_\chi = \{n \in \NN : \chi(n) = 1\}$.
For example, for $k \in \NN\cup\{\infty\}$ one can take $\chi_k$ such that
$$\chi_k(p^n) = \begin{cases}
1 & \textnormal{ if } n \le k,\\
0 & \textnormal{ otherwise.}\end{cases}$$
Then $\NN_k := \NN_{\chi_k}$ is the set of $(k+1)$-free integers if $k \in \NN$
and $\NN_\infty=\NN$.

\begin{cor}\label{cor:weak-composite}
Let $\chi$ as above.
Assume the setting of part~(\ref{c-thm:prime}) of Theorem~\ref{thm:prime}.
Then the sets
$$ \{n \in \NN \;|\; n \in \NN_\chi \textnormal{ and } a(tn^2) > 0 \}
\textrm{ and } \{n \in \NN \;|\; n \in \NN_\chi \textnormal{ and } a(tn^2) < 0 \}$$
have equal positive Dedekind-Dirichlet densities, that is,
both are precisely half of the density of the set
$$ \{n \in \NN \;|\; n\in \NN_\chi \textnormal{ and } a(tn^2) \neq 0 \}.$$
\end{cor}

\begin{proof}
Note that without loss of generality we can assume $a(t)>0$.
Define the arithmetic function $\psi:\NN \to \{-1,0,1\}$ as follows:
$$ \psi(n) := \chi(n) \cdot \begin{cases}
1 & \textnormal{ if } a(tn^2) > 0,\\
-1 & \textnormal{ if } a(tn^2) < 0,\\
0 & \textnormal{ if } a(tn^2) =0.\end{cases}$$
Equation~\eqref{eq:mult} implies that $\psi$ is a multiplicative function.
Note that $\PP_{>0} = \{p \in \PP : \psi(p)=1\}$, $\PP_{<0} = \{p \in \PP : \psi(p)=-1\}$,
and $\PP_{=0} = \{p \in \PP : \psi(p)=0\}$.
Theorem~\ref{thm:prime} shows that these sets are weakly regular and allows us
to conclude due to Proposition~\ref{prop:weak-equi}.
\end{proof}

\begin{cor}\label{cor:reg-composite}
Let $\chi$ as above.
Assume the setting of part~(\ref{d-thm:prime}) of Theorem~\ref{thm:prime}.
Then the sets
$$ \{n \in \NN \;|\; n \in \NN_\chi \textnormal{ and } a(tn^2) > 0 \}
\textrm{ and } \{n \in \NN \;|\; n \in \NN_\chi \textnormal{ and } a(tn^2) < 0 \}$$
have equal positive natural densities, that is,
both are precisely half of the density of the set
$$ \{n \in \NN \;|\; n\in \NN_\chi \textnormal{ and } a(tn^2) \neq 0 \}.$$
\end{cor}

\begin{proof}
The proof proceeds precisely as that of Corollary~\ref{cor:weak-composite}, except
that in the end we appeal to Proposition~\ref{prop:consequence-Delange}.
\end{proof}

\bibliographystyle{amsalpha}

\end{document}